\documentclass[graybox]{svmult}

\usepackage{mathptmx}       
\usepackage{helvet}         
\usepackage{courier}        
\usepackage{type1cm}        
\usepackage{makeidx}         
\usepackage{graphicx}        
\usepackage{multicol}        
\usepackage[bottom]{footmisc}

\makeindex             

\usepackage{amsmath,amssymb,rotating,wasysym,mathrsfs}

\usepackage{tikz}

\newcommand{\bm}[1]{\mbox{\boldmath $#1$}}

\def\be{\begin{equation}}
\def\ee{\end{equation}}
\def\bea{\begin{eqnarray}}
\def\eea{\end{eqnarray}}
\def\bean{\begin{eqnarray*}}
\def\eean{\end{eqnarray*}}

\def\espaitemps{({\cal V},g)}
\def\varietat{{\cal V}}

\def\Journal#1#2#3#4{{#1} {\bf #2}, #3 (#4)}

\def\GRG{\em Gen. Rel. Grav.}

\def\PRL{\em Phys. Rev. Lett.}

\def\ii{\mbox{\rm I\!I}}

\begin{document}

\title*{Umbilical-Type Surfaces in Spacetime}
\author{Jos\'e M. M. Senovilla}
\institute{Jos\'e M. M. Senovilla \at F\'{\i}sica Te\'orica, Universidad del Pa\'{\i}s Vasco, Apartado 644, 48080 Bilbao (Spain) \email{josemm.senovilla@ehu.es}
}
%
%
\maketitle

\abstract{A spacelike surface $S$ immersed in a 4-dimensional Lorentzian manifold will be said to be umbilical along a direction $N$ normal to $S$ if the second fundamental form along $N$ is proportional to the first fundamental form of $S$. In particular, $S$ is pseudo-umbilical if it is umbilical along the mean curvature vector field $H$, and (totally) umbilical if it is umbilical along all possible normal directions. The possibility that the surface be umbilical along the unique normal direction orthogonal to $H$ --- ``ortho-umbilical'' surface--- is also considered.\newline\indent
I prove that the necessary and sufficient condition for $S$ to be umbilical along a normal direction is that two independent Weingarten operators (and, a fortiori, all of them) commute, or equivalently that the shape tensor be diagonalizable on $S$. The umbilical direction is then {\em uniquely} determined. 
This can be seen to be equivalent to a condition relating the normal curvature and the appropriate part of the Riemann tensor of the spacetime. In particular, for conformally flat spacetimes (including Lorentz space forms) the necessary and sufficient condition is that the normal curvature vanishes. 
Some further consequences are analyzed, and the extension of the main results to arbitrary signatures and higher dimensions is briefly discussed.}

\section{Introduction}
\label{sec:Int}
Spacelike surfaces play a crucial role in gravitational physics (such as in General Relativity and in any other geometrical theory based on a Lorentzian manifold), specially those which are (marginally) (outer) trapped and closed ---compact with no boundary---, see subsection \ref{subsec:H}. 

A few years ago I presented a complete local classification of spacelike surfaces in 4-dimensional Lorentzian manifolds \cite{S4}, and discussed its generalization to arbitrary dimensions. The classification was carried out according to the {\em extrinsic} properties of the surface: it is an algebraic classification based, at each point, on the properties of two independent Weingarten operators. Specifically, I used two {\em null} Weingarten operators $A_{\ell}$ and $A_k$ (where $\ell$ and $k$ are the two independent null vectors fields orthogonal to the surface, see below.)
Each Weingarten operator is a self-adjoint matrix which can be readily classified algebraically according to the signs of their (real) eigenvalues. This produces 8 different types for each matrix, and therefore 64 types of points for generic spacelike surfaces.

To my surprise, this was not enough for a complete classification, and I had to introduce an extra parameter to each point taking into account the {\em relative orientation} of the two null Weingarten operators at the chosen point. 
As far as I know, this extra parameter had not been considered in the literature ---not even for the proper Riemannian case. 
The parameter can be chosen as the angle between the two orthonormal (ON) eigen-bases for $A_\ell$ and $A_k$. Therefore, it takes values on a finite closed interval of $\mathbb{R}$.
Actually, one can prove \cite{S4} that the parameter is simply related to the commutator
$$
\left[A_k,A_\ell\right]
$$
of the two null Weingarten operators.

The meaning and interpretation of this parameter became an important open question, and the goal of the present paper is to answer it. The main theorems to be proven are the following:
\begin{svgraybox}
\begin{theorem}\label{th:1}
Consider a spacelike surface $S$ immersed in a 4-dimensional Lorentzian manifold $\espaitemps$. The necessary and sufficient condition for $S$ to be umbilical along a normal direction is that two independent Weingarten operators ---and, a fortiori, all of them--- commute.

The umbilical direction is then uniquely determined ---unless the surface is totally umbilical.
\end{theorem}
\end{svgraybox}
This happens to be equivalent to the condition that the shape tensor be diagonalizable on $S$.

\begin{svgraybox}
\begin{theorem}\label{th:2}
The necessary and sufficient condition for $S$ to be umbilical along a normal direction is that the normal curvature of $S$ equals the ``tangent-normal" part of the Riemann tensor of $\espaitemps$.
\end{theorem}
\begin{corollary}\label{coro:1}
In particular, for locally conformally flat $\espaitemps$ (including Lorentz space forms) the neccessary and sufficient condition for $S$ to be umbilical along a normal direction is that the normal curvature vanishes.
\end{corollary}
\end{svgraybox}
A precise formulation of Theorem \ref{th:2} is presented in Remark \ref{reformulation} after the necessary notions and notations have been introduced.

There are several other interesting consequences of these theorems, as well as explicit formulas for the umbilical direction. These will be presented in section \ref{sec:consequences}. In subsection \ref{subsec:G}, I introduce a new vector field $G$, normal to the surface, which characterizes the umbilical property and, together with the traditional mean curvature vector $H$, defines the main properties of the surface.

The main results extend to {\em non-null} surfaces in 4-dimensional semi-Riemannian manifolds of {\em arbitrary} signature as I discuss succinctly at the end of the paper. The higher dimensional case is, however, an open problem.

\section{Basic Concepts and Notation}
\label{sec:basics}
Let $\espaitemps$ be a $4$-dimensional, oriented and time-oriented, Lorentzian manifold with metric tensor $g$ of signature $(-,+,+,+)$. At every $x\in \varietat$, the isomorphism between tangent vectors and one-forms, that is, between $T_{x}\varietat$ and $T^{*}_{x}\varietat$ is denoted as follows
$$
\begin{array}{rcc}
\flat \, : T_{x}\varietat & \longrightarrow & T^{*}_{x}\varietat \\
v & \longmapsto & v^{\flat}
\end{array}
$$
and defined by $v^{\flat}(w) = g(v,w)$, $\forall w\in T_{x}\varietat $. Its inverse map is denoted by $\sharp$. These maps extend naturally to the tangent and co-tangent bundles.

An immersed surface is given by the pair $(S,\Phi)$ where $S$ is a 2-dimensional manifold and $\Phi : S \longrightarrow \varietat$ is an immersion. Such an $S$ does not have to be necessarily orientable. However, as the computations herein presented will be local, I will tacitly assume ---without loss of generality--- that $S$ is embedded and oriented. For instance, given that any point of an immersed $S$ has an open neighborhood which can be identified with its image in $\espaitemps$, to avoid confusion and unnecessary complications in the notation $S$ will be identified with its image $\Phi (S)$ in $\varietat$.
 
The {\it first fundamental form} of $S$ in $\varietat$ is simply $\bar g\equiv \Phi^{*}g$, where $\Phi^{*}$ is the pullback of $\Phi$. From now on $\bar g$ will be assumed to be positive definite on $S$, which implies that every tangent vector in $T_{x}S$ $\forall x\in S$ is spacelike and then $S$ is said to be {\em spacelike}.
For such an $S$, at any $x\in S$ one has the orthogonal decomposition
$$
T_{x}\varietat =T_{x}S\oplus T_{x}S^{\perp}\, .
$$
 Let $\mathfrak{X}(S)$ (respectively $\mathfrak{X}(S)^{\perp}$) denote the set of smooth vector fields tangent (resp. orthogonal) to $S$. In what follows, and for the sake of brevity, I will often give definitions and properties on $\mathfrak{X}(S)$, but they of course have always a previous, more fundamental, version on each $T_{x}S$. Thus, for instance, the volume element 2-form associated to $(S,\bar g)$ ---denoted here by by $\bar\epsilon$---, together with the volume element 4-form $\epsilon$ of $\espaitemps$, induces a volume element 2-form on each $T_{x}S^{\perp}$, denoted by $\epsilon^{\perp}$. The corresponding Hodge dual operator (see e.g. \cite{BHMS}) is written and defined as 
$$
\star^{\perp}N \equiv (i_N\epsilon^\perp )^\sharp , \,\,\, \forall N\in \mathfrak{X}(S)^\perp \, .
$$
The vector field $\star^{\perp}N$ defines the unique normal direction in $\mathfrak{X}(S)^\perp$ orthogonal to the normal $N\in \mathfrak{X}(S)^\perp$.

The surface $S$ with the first fundamental form is a Riemannian manifold $(S,\bar g)$. As is well known, its {\it Levi-Civita connection} $\bar\nabla$, with $\overline\nabla \bar g =0$ and no torsion, can be defined as \cite{O,Kr}
$$
\overline\nabla_{X}\, Y\equiv \left(\nabla_{X}\, Y\right)^{T} , \hspace{1cm} \forall X, Y\in \mathfrak{X}(S) 
$$
where $\nabla$ is the canonical connection of $\espaitemps$.\footnote{There is a long standing tradition among mathematicians who study submanifolds to use the opposite convention, that is, $\nabla$ for the inherited connection and $\bar\nabla$ for the background connection. I stress this point here in the hope that this will avoid any possible confusion.}

The {\em normal connection} $D$ acts, in turn, on $\mathfrak{X}(S)^\perp$
$$
D_{X} : \mathfrak{X}(S)^\perp \longrightarrow \mathfrak{X}(S)^\perp
$$
for $X\in \mathfrak{X}(S)$, and is given by the standard definition \cite{O}
$$
D_X N \equiv \left(\nabla_X N \right)^\perp , \hspace{5mm} \forall N \in \mathfrak{X}(S)^\perp \hspace{5mm} \forall X\in \mathfrak{X}(S).
$$

\subsection{Extrinsic geometry: 2nd fundamental forms and Weingarten operators.}
The basic extrinsic object is $
\ii : \, \mathfrak{X}(S) \times \mathfrak{X}(S) \longrightarrow \mathfrak{X}(S)^{\perp}$, called the {\em shape tensor} or {\em second fundamental form tensor} of $S$ in $\varietat$ and defined by \cite{O,Kr}
$$
-\ii(X, Y)\equiv \left(\nabla_{X}\, Y\right)^{\perp} =\nabla_{X}\, Y - \overline\nabla_{X}\, Y \, , \hspace{1cm} \forall X, Y\in \mathfrak{X}(S) 
$$
(observe the choice of sign, that may be unusual in some contexts and is actually opposite to  \cite{O,Kr}).
$\ii$ contains the information concerning the ``shape'' of $S$ within $\varietat$ along {\em all} directions normal to $S$. Notice that $$\ii(X, Y)=\ii(Y,X).$$

Given any normal direction $N\in \mathfrak{X}(S)^{\perp}$, the {\em second fundamental form} of $S$ in $\espaitemps$ relative to $N$ is the 2-covariant symmetric tensor field on $S$ defined by
$$
{K}_{N}(X,Y)\equiv g \left(N,\ii (X,Y)\right),
\hspace{1cm} \forall X, Y\in \mathfrak{X}(S)\, .
$$
The {\em Weingarten operator} 
$$A_N : \mathfrak{X}(S)\longrightarrow \mathfrak{X}(S)$$ 
associated to $N\in \mathfrak{X}(S)^{\perp}$ is defined by
$$
A_N(X) \equiv \left(\nabla_X N \right)^T , \hspace{1cm} \forall X \in \mathfrak{X}(S).
$$
Observe that $$\bar g (A_N(X),Y)={K}_{N}(X,Y), \hspace{3mm}\forall X, Y\in \mathfrak{X}(S),$$
hence, at each $x\in S$, $A_N|_x$ is a self-adjoint linear transformation on $T_x S$. As such, it is always diagonalisable over $\mathbb R$.

\subsection{Special bases on $ \mathfrak{X}(S)^{\perp}$.}
$S$ having co-dimension 2, there are two {\em independent} normal vector fields on $S$. They can be appropriately chosen to form an ON basis on $\mathfrak{X}(S)^{\perp}$, in which case I will denote them by $u,n\in \mathfrak{X}(S)^{\perp}$, with 
$$
g(n,n)=-g(u,u)=1, \,\, \, g(u,n)=0 \, .
$$
 Of course, any two such ON bases are related by a Boost (Lorentz transformation):
\begin{equation}
\left(\begin{array}{c} u' \\ n' \end{array}\right)=\left(\begin{array}{cc} \cosh\beta &\sinh\beta \\ \sinh\beta & \cosh\beta\end{array}\right)\left(\begin{array}{c} u \\ n\end{array}\right)\label{free0}
\end{equation}
where $\beta$ is a smooth function on $S$.

The two independent normal vector fields can also be chosen to be null (and future-pointing say), and I will denote these by $k,\ell \in \mathfrak{X}(S)^{\perp}$,
so that $$g(\ell,\ell)=g(k,k) =0, \hspace{3mm} g({\ell},{k})=-1$$
the last of these being a convenient normalization condition. Observe that, to any ON basis $\{u,n\}$ on $\mathfrak{X}(S)^{\perp}$, one can associate a null basis given by $\sqrt{2}\, \ell = u+n$ and $\sqrt{2}\,  k = u-n$.
The previous boost  freedom becomes now simply
\begin{equation}
\ell \longrightarrow {\ell}'=e^\beta {\ell}, \hspace{1cm}
{k} \longrightarrow {k}'=e^{-\beta} {k} \label{free}
\end{equation}
so that the two independent null {\em directions} are uniquely determined on $S$.

The orientations of $\espaitemps$ and of the imbedded surface $(S,\bar g)$ will be chosen such that the operator $\star^\perp$ acts on the previous bases as follows
 $$\star^{\perp}u=n, \,\,\, \star^{\perp}n=u; \hspace{15mm}
\star^{\perp}\ell =\ell , \,\,\, \star^{\perp}k =-k .$$

\subsection{The mean curvature vector field $H$ and its causal character}
\label{subsec:H}
The shape tensor decomposes as
\be
\ii (X,Y)=-{K}_{k}(X,Y)\,\, \ell -{K}_{\ell}(X,Y)\,\,  k \label{shape}
\ee
in a null basis, or as
$$
\ii (X,Y)= -K_u (X,Y) \,\, u + K_n (X,Y) \,\, n \eqno{(3')}
$$
in any ON basis, $\forall X, Y\in \mathfrak{X}(S)$. Note that these formulae are invariant under the boost freedom (\ref{free0}-\ref{free}).
 The {\em mean curvature vector} $H\in \mathfrak{X}(S)^\perp$ is defined as the trace of the shape tensor with respect to $\bar g$, or explicitly 
\be
{H}= -(\mbox{tr}\,{A}_{k})\,\, \ell - (\mbox{tr}\, {A}_{\ell})\,\, k \, \label{mean}
\ee
in a null basis, or
$$
{H}= -(\mbox{tr}\,{A}_{u})\,\, u + (\mbox{tr}\, {A}_{n})\,\, n \eqno{(4')}
$$
in ON bases. In the physics literature, each component of $H$ along a particular normal direction $g(H, N)$= tr$A_{N}$ is termed ``expansion along $N$'' of $S$ \cite{BEE,Kr,S4}. In particular, tr$A_{\ell}$ and tr$A_{k}$ are called the {\em null expansions}.

 Notice that $H$ and $$\star^\perp H=-(\mbox{tr}\,{A}_{k})\,\, \ell + (\mbox{tr}\, {A}_{\ell})\,\, k= -(\mbox{tr}\,{A}_{u})\,\, n + (\mbox{tr}\, {A}_{n})\,\, u$$ are well-defined, they are invariant under the boost gauge freedom, and actually under arbitrary changes of basis. Observe also that $\star^\perp H$ is a (generically unique) direction with vanishing expansion: tr$A_{\star^\perp H}=0$. This fact is important in physics sometimes.
 
A very important type of surface (or submanifolds) in Riemannian geometry are the {\em minimal} ones. They are characterized by the vanishing of the mean curvature vector, that is, by the condition $H=0$. Observe, however, that in Riemannian geometry any vector can only have either zero or positive norm and, hence, the only distinguished case for $H$ is when it vanishes: the minimal surfaces. In semi-Riemannian geometry, though, vectors such as $H$ can realize all signs for $g(H,H)$, in particular $H$ can be timelike ($g(H,H)<0$) or null ($g(H,H)=0$), in the last case with $H\neq 0$. And these new cases provide new types of surfaces (and submanifolds) in the Lorentzian case.

Actually, the most important surfaces in Gravitation are defined according to such causal orientations of $H$. For instance, the simple condition 
$$
H^\flat \wedge (\star^\perp H)^\flat =0
$$
is equivalent to saying that $H$ is null everywhere on $S$. These will be called {\em null $\varhexstar$-surfaces} due to the nomenclature introduced in \cite{S4}, see also \cite{S7,Kh}. In the mathematical literature, surfaces with a null $H$ were considered for instance in \cite{R1,R2} for the Minkowski space-time under the name of ``pseudo-minimal'' or ``quasi-minimal'' surfaces, see also \cite{Chen}. Among  null $\varhexstar$-surfaces an important case is when $H$ (and hence $\star^\perp H$) points along one of the null directions $\ell$ or $k$ everywhere, then they are called {\em marginally outer trapped surfaces} (MOTS) (also called {\em null dual}). They have received a great deal of attention lately, in particular concerning their stability \cite{AMS,AMS1,AM,CM} which leads to the study of an elliptic operator similar to the stability operator for minimal surfaces. 

If in addition to $H$ pointing along one of the null normal directions its causal orientation does not change on $S$, that is to say, it is everywhere null future or everywhere null past, then $S$ is called a {\em marginally (future or past) trapped surface}, \cite{HE,Wald,MS,S4,S7}. For references concerning this type of surfaces in the mathematical literature one can consult the recent book \cite{Chen}. 

If on the other hand $H$ keeps its future (or past) causal orientation everywhere on $S$ (but it can change from null to timelike from point to point), then the surface is usually called {\em weakly (future or past) trapped} \cite{HE,Wald,MS,S4,S7}. If $H$ is also timelike non-zero all over $S$ then it is said to be (future or past) trapped \cite{P2,HE,HP,Wald,BEE,S4,Chen}. The concept of closed trapped surface ---here closed means compact with no boundary--- was introduced by Penrose \cite{P2} in a seminal paper where the first modern-type singularity theorem was proven. It was immediately realized that the concept of trapping is essential in many important developments concerning gravitational collapse and the formation of black holes, such as the singularity theorems \cite{HE,HP,S,GS}, the so-called ``cosmic censorship conjecture'' \cite{P6} together with the related subject of iso-perimetric or Penrose inequalities \cite{P5,G,G2,M}, and the hoop conjecture \cite{MTW,S6,CGP}.

\subsection{The extrinsic vector field $G$}
\label{subsec:G}
One can also define another normal vector field $G \in \mathfrak{X}(S)^{\perp}$ by using a second invariant of the matrices $A_N$. Unfortunately, there are no other {linear} invariants. In spite of that, for each $N \in \mathfrak{X}(S)^{\perp}$ one can set
$$
\sigma^2_N \equiv (\mbox{tr} A_N)^2-4\det {A}_{N}
$$
which is called the {\em shear} along $N$ \cite{HE,Wald}, and can also be expressed as the square of the difference of the two eigenvalues of $A_N$. An alternative formula is
$$
\det \left(A_N-\frac{1}{2}\mbox{tr} A_N  {\bf 1} \right)=-\frac{1}{4} \sigma^2_N .
$$
It should be noted that the matrix $A_N-\frac{1}{2}\mbox{tr} A_N  {\bf 1}$ is traceless, and therefore its two eigenvalues have opposite signs: $\sigma_N^{2}/4$ is the square of either of them. However, fixing the sign of $\sigma_N$ so that it becomes a differentiable function on $S$ is not free from ambiguities.\footnote{If one chooses, say, $\sigma_N$ to be the positive root of $\sqrt{\sigma_N^2}$ then it may fail to be differentiable at points where the two eigenvalues of $A_N$ coincide, that is, at points where $\sigma_N=0$. Of course, one can always set an ``initial'' condition for $G|_x$ at any point on $x\in S$, and then the differentiable solution for the vector field $G$ is fixed. Nevertheless, this initial condition is arbitrary.} Whatever the signs chosen, I set by definition
$$
\fbox{$\displaystyle{
G \equiv \sigma_k\,  \ell + \sigma_{\ell}\, k}$} 
$$
The two possible signs for each of $\sigma_k,\sigma_\ell$ provide four distinct possibilities for $G$ which define, by ignoring overall orientations, two orthogonal directions. However, these two directions can always be described by $G$ together with
$$\star^\perp G = \sigma_k\,  \ell - \sigma_{\ell}\, k$$
independently of the chosen signs. Observe that both $G$ and $\star^\perp G$ 
are invariant under the boost freedom (\ref{free}). 

It is important to remark that, due to the non-linearity of the invariants $\sigma^2_N$,
$$G\neq \sigma_u u-\sigma_n n$$ 
in general. There are points where the equality holds, and they will turn out to be precisely the umbilical points. 

The vector field $G$ is intimately related to the umbilical properties of a surface $S$, as I am going to prove presently.

\subsection{The normal connection one-form $s$}
For a fixed ON basis on $\mathfrak{X}(S)^\perp$, a one-form $s \in \Lambda^1(S)$ is defined by
$$
s (X) \equiv -g(u, D_X n) = g(D_X u, n), \hspace{1cm} \forall X \in \mathfrak{X}(S).
$$
For $\sqrt{2}\ell = u+n$ and $\sqrt{2} k =u-n$ one can alternatively write
$$
s (X) \equiv -g(k, D_X \ell) = g(D_X k, \ell), \hspace{1cm} \forall X \in \mathfrak{X}(S).
$$
Therefore, for all $X\in \mathfrak{X}(S)$
$$
D_X u =s(X) n, \, \, D_X n = s(X) u; \hspace{5mm} D_X\ell =s(X)\ell , \, \, D_X k = -s(X) k .
$$
Observe that $s$ is not invariant under boost rotations (\ref{free0}) or (\ref{free}). Actually, $s$ is a ``connection" and transforms as $s '(X) = s(X) + X(\beta)$ under those transformations, or simply
$$
s ' =s +d\beta.
$$
It follows that $ds=ds'$ is invariant and well-defined. It will be proven in the next subsection that this is actually related to the normal curvature on $S$, see formula (\ref{normalR}), confirming the connection character of $s$. In the mathematical literature on riemanniana geometry $S$ is sometimes called the {\em third fundamental form} of $S$ in $\espaitemps$, see e.g. \cite{YI}.

\subsection{Curvatures: Gauss and Ricci equations.}
The intrinsic curvature for $(S,\bar g)$ has the usual definition
$$
\overline{R} (X,Y) Z \equiv \overline\nabla_X\overline\nabla_Y Z-\overline\nabla_Y\overline\nabla_X Z -\overline\nabla_{[X,Y]} Z, \hspace{4mm} \forall X,Y,Z \in \mathfrak{X}(S).
$$
Similarly, the normal curvature is defined on $S$ by
$$
R^\perp (X,Y) N \equiv D_X D_Y N - D_Y D_X N -D_{[X,Y]} N, \hspace{4mm} \forall X,Y\in \mathfrak{X}(S), \hspace{4mm} \forall N\in \mathfrak{X}(S)^\perp .
$$
A simple calculation provides
\be
R^\perp (X,Y) N = ds (X,Y)\,  \star^\perp\! N , \hspace{1cm} \forall X,Y \in \mathfrak{X}(S), \,\, \, \, \forall N\in \mathfrak{X}(S)^\perp . \label{normalR}
\ee
This justifies that $s$ describes the normal connection and that $ds$ defines its curvature.

The Gauss equation relating the curvatures of $(S,\bar g)$ and $\espaitemps$ can be written as

\be
R(W,Z,X,Y) = \overline{R}(W,Z,X,Y)+g\left(\ii (X,Z),\ii (Y,W) \right)-g\left(\ii (Y,Z),\ii (X,W) \right)
\label{gauss}
\ee
for all $X,Y,Z,W\in \mathfrak{X}(S)$, where I use the notation \footnote{Notice the sign convention, which may not coincide with the preferred one for everybody.}
$$
R(W,Z,X,Y) \equiv g(W,R(X,Y)Z)
$$
and analogously for $\overline{R}$. 
However, as $S$ is 2-dimensional its curvature is uniquely determined by its Gaussian curvature $K(S)$. Therefore, the previous relation can be written as a single scalar equation. 
To that end, let me define a new extrinsic object, quadratic in the shape tensor $\ii$, as follows.
For any ON basis $\{e_1,e_2\}$ in $\mathfrak{X}(S)$, set by definition 
$$\mathbb{J}(X,Y)\equiv \sum_{i=1}^2 g\left(\ii(e_i,X),\ii(e_i,Y) \right), \hspace{1cm} \forall X,Y \in \mathfrak{X}(S)\, .$$ 
$\mathbb{J}(X,Y)$ is a 2-covariant symmetric tensor field on $S$. Then, define $B\, : \mathfrak{X}(S)\rightarrow \mathfrak{X}(S)$ by
$$
g(B(X),Y) \equiv \mathbb{J} (X,Y), \hspace{1cm} \forall X,Y \in \mathfrak{X}(S) \, .
$$
$B$ is sometimes called {\it the Casorati operator} of $S$ in $\espaitemps$ \cite{Chen}, and has been mainly studied in the Riemannian case, see e.g. \cite{DHV,HKV} and references therein. In the Lorentzian case under consideration in this paper, a straightforward calculation allows one to check that $B$ is the anticommutator of the two null Weingarten operators:
\be
B=-\left\{A_k ,A_\ell\right\}.
\label{B}
\ee
Once more, let me remark that $B$ is invariant under the boost freedom (\ref{free}). Observe furthermore that
$$
\mbox{tr} B =g(\ii ,\ii)
$$
which is sometimes called the Casorati curvature \cite{Chen}.

With the previous notation the Gauss equation (\ref{gauss}) becomes
\be
2\, K(S) = {\cal S} -4\,  \mbox{Ric} (\ell , k) +2 R(\ell,k,\ell,k)+ g(H,H) -\mbox{tr}\, B \label{gauss2}
\ee 
where $\mbox{Ric}$ and ${\cal S}$ are the Ricci tensor and the scalar curvature of $\espaitemps$.


With regard to the Ricci equation, relating the normal curvature $R^\perp$ with the tangent-normal part of the spacetime curvature $R$ on $S$, one can write
\bean
\left(R(X,Y) N \right)^\perp &=& \ii \left(X,A_N (Y) \right)-\ii  \left(Y,A_N(X)\right)+R^\perp(X,Y)N\\
&=& \ii \left(X,A_N (Y) \right)-\ii  \left(Y,A_N(X)\right)+ds (X,Y) \star^\perp\! N
\eean
for all $X,Y\in \mathfrak{X}(S)$ and for all $N,M \in \mathfrak{X}(S)^\perp$, where in the last equality I have used (\ref{normalR}). An alternative possibility, which will reveal itself as very useful in the sequel, is
\be
R(M,N,X,Y)= g\left(\left[A_M , A_N\right](Y),X\right)+ds (X,Y)\, \,  g(\star^\perp N,M), 
\label{ricci}
\ee
for all $X,Y\in \mathfrak{X}(S)$ and all $N,M \in \mathfrak{X}(S)^\perp $.

\section{Umbilical-type, pseudo-umbilical, and related surfaces}
The concept of umbilical point is classical in semi-Riemannian geometry. When the co-dimension of a submanifold is higher than one, then there are several possible directions along which a point can be umbilic. Specifically
\begin{definition}[Umbilical points on $S$]
A point $x\in S$ is called {\em umbilical} with respect to $N|_x \in T_x^\perp S$ (or simply {$N$-umbilical}) if the corresponding Weingarten operator is proportional to the Identity
$$
A_N|_x = \frac{1}{2} F\, {\bf 1} .
$$
Obviously, in that case $F = \mbox{tr} A_N|_x = g(H,N|_x)$ necessarily. An equivalent characterization is
$$K_N |_x =  \frac{1}{2} g(H,N|_x)\,  \bar g |_x.$$
\end{definition}
\begin{definition}[$N$-Umbilical surfaces]
Thus, $S$ is said to be {\em umbilical along a vector field} $N\in \mathfrak{X}(S)^\perp$ if 
\be
A_N =  \frac{1}{2} g(H,N) {\bf 1}\label{umb}
\ee
or equivalently, if
$K_N = \frac{1}{2} g(H,N) \, \bar g$.
\end{definition}
This concept was studied in the Riemannian case many years ago under some special circumstances, e.g. \cite{CY1,CY2}, see \cite{Chen} for the general semi-Riemannian case.

Observe that {\em minimal surfaces}, that is those with zero mean curvature vector $H=0$, can be considered as a limit case of $N$-umbilical  surfaces only in the case that the whole Weingarten operator vanishes $A_N=0$. This motivates the following definition \cite{S4}.
\begin{definition}[$N$-subgeodesic surface]
A spacelike surface $S$ is called {\em $N$-subgeodesic}, for $N\in \mathfrak{X}(S)^\perp$, if $A_{\star^\perp N}=0$.
\end{definition}
This means that {\em any} geodesic $\gamma : I\subset \mathbb R \longrightarrow S$ of the surface $(S,\bar g)$ is a sub-geodesic \cite{Scho} with respect to $N$ on the spacetime $\espaitemps$: its tangent vector $\gamma\, '$ satisfies the relation 
$$\nabla_{\gamma\, '} \gamma\, ' =f N$$
where the function $f$ on $\gamma$ is fully determined by the relation 
$f N= \ii (\gamma\, ', \gamma\, ')$.

Obviously, a surface is sub-geodesic with respect to two independent normal vector fields $N$ and $M$ ($N^\flat \wedge M^\flat\neq 0$) if and only if it is {\em totally geodesic} ($\ii =0$) \cite{O}, or equivalently, if and only if $A_N =0, \, \, \forall N\in \mathfrak{X}^\perp (S)$.

\begin{remark}
\label{A-sub}
In traditional Riemannian geometry there is the concept of {\it first normal space} ${\cal N}_1$ for immersed submanifolds $S$, defined at each $p\in S$ by 
$${\cal N}_1=\mbox{Span}\{\ii (X,Y); \, X, Y \in T_pS\}.$$
This generalizes immediately to the general semi-Riemannian case, and then $N$-subgeodesic surfaces have dim${\cal N}_1\leq 1$, because $\ii (X,Y)^\flat \wedge  N^\flat =0$ for all $X,Y\in \mathfrak{X}(S)$. Actually, $N$-subgeodesic surfaces are characterized by
$$
\ii (X,Y)^\flat \wedge \ii (Z,W)^\flat =0 \hspace{1cm} \forall X,Y,Z,W \in \mathfrak{X} (S) 
$$
and then the direction $N$ can be determined by computing $\ii (X,X)$ for any $X\in \mathfrak{X} (S)$ such that $\ii (X,X)\neq 0$---and whenever $S$ is not totally geodesic, of course.
In other words, all possible second fundamental forms, or all the Weingarten operators, are proportional to each other as follows from the fact that $\ii (X,Y) = K(X,Y) N$ for some fixed \footnote{Up to proportionality factors, this $K$ coincides with $K_{N}$ if $N$ is non-null. If $N$ is null, then $N$ can be chosen to be either $\ell$ or $k$, and $K$ is $-K_k$ or $-K_\ell$, respectively.} rank-2 symmetric covariant tensor field $K$ in $S$. Given that all submanifolds with  dim${\cal N}_1\leq 1$ are trivial $A$-submanifolds, a concept introduced in \cite{Chen0} for Riemannian manifolds ---see also \cite{D,H,Rou} and references therein for some simple Lorentzian cases---, then $N$-subgeodesic surfaces are in particular trivial $A$-submanifolds. 
\qed
\end{remark}

A standard possibility for umbilical surfaces in sub-manifolds with co-dimension higher than one is that the umbilical direction is given by the mean curvature vector. These are called pseudo-umbilical surfaces \cite{Chen}.
\begin{definition}[Pseudo-umbilical surface]
$S$ is said to be {pseudo-umbilical} if it is umbilical with respect to $N=H$, so that $$A_H =  \frac{1}{2} g(H,H) \, \bf{1} .$$
\end{definition}
In Riemannian geometry this kind of submanifolds have been studied since long ago, see e.g. \cite{Ot} and specifically  \cite{YI} for the co-dimension 2 situation. Probably the first study in the semi-Riemannian  case was performed in \cite{R0} and then only much later in \cite{KK,Sun}. Some results concerning pseudo-umbilical submanifolds in semi-Riemannian geometry can be consulted in \cite{AER,BE,Cao,HJN,SP,Chen}, not much of it specific for Lorentzian geometry. Thus, as far as I am aware, very few things are known for pseudo-umbilical surfaces in general Lorentzian manifolds.

Less common is the idea of $S$ being umbilical along the {\em unique} direction orthogonal to $H$ in $\mathfrak{X}(S)^\perp$. Actually,  this idea does not appear to have been considered previously, so that the following definition is new and I made the name for this type of surface up ---maybe not too skillfully.
\begin{definition}[Ortho-umbilical surface]
A surface $S$ will be called {\em ortho-umbilical} if it is umbilical with respect to $N=\star^\perp H$, so that 
$$A_{\star^\perp H}=0.$$
\end{definition}
As a matter of fact, for the case of co-dimension two under consideration one can prove the following equivalence between ortho-umbilical and $N$-subgeodesic surfaces.
\begin{proposition}
\label{equiv}
The following conditions are equivalent for a non-minimal spacelike surface $S$ in $\espaitemps$:
\begin{enumerate}
\item $S$ is ortho-umbilical
\item $S$ is $N$-subgeodesic for some $N\in \mathfrak{X}(S)^\perp$
\item $S$ is $H$-subgeodesic
\end{enumerate}
\end{proposition}
\begin{proof}[of Theorem \ref{th:1}]
\smartqed

1 {\Large $\Longrightarrow$} 3 \hspace{0.2cm} Assume $S$ is ortho-umbilical and $H\neq 0$. This means that $A_{\star^\perp H}=0$ which is the definition of $H$-subgeodesic.

3 {\Large $\Longrightarrow$} 2 \hspace{0.2cm} Trivial

2 {\Large $\Longrightarrow$} 1 \hspace{0.2cm} If $S$ is $N$-subgeodesic then $A_{\star^\perp N}=0$, so that $K_{\star^\perp N}=0$ too. Let $M\in \mathfrak{X}(S)^\perp$ be any normal vector field such that $N^\flat \wedge M^\flat\neq 0$, ergo span$\{N,M\}=\mathfrak{X}(S)^\perp$. In the basis $\{N,M\}$ one obviously has
$$
\ii (X,Y) = K_1(X,Y) N + K_2(X,Y) M, \hspace{1cm} \forall X,Y \in \mathfrak{X} (S) 
$$
for some rank-2 covariant symmetric tensor fields $K_1,K_2$ on $S$. As a matter of fact $K_1$ and $K_2$ are determined by $K_N$ and $K_M$ as follows: $K_N = g(N,N)K_1 + g(N,M)K_2 $ and $K_M= g(N,M) K_1 + g(M,M)K_2 $. The property $K_{\star^\perp N}=0$ implies, on using that by definition $N$ and $\star^\perp N$ are mutually orthogonal, that 
$$
g(\star^\perp N,M) K_2 (X,Y)=0 \hspace{1cm} \forall X,Y \in \mathfrak{X} (S) \, .
$$
Notice, however, that $g(\star^\perp N,M)\neq 0$ because the unique direction in $ \mathfrak{X}^\perp (S) $  orthogonal to $\star^\perp N$ is actually $N$, and the choice of $M$ prevents that $M$ and $N$ be proportional. Thus, necessarily $K_2 =0$ implying that $\ii (X,Y) = K_1(X,Y) N $ for all $X,Y \in \mathfrak{X} (S)$ and, as a consequence, that 
$$H=\mbox{tr}A_1 N$$
where $A_1 : \mathfrak{X} (S) \rightarrow \mathfrak{X} (S)$ is an operator {\it \`a la} Weingarten associated to $K_1$, that is to say, defined by $\bar g(A_1(X),Y)=K_1(X,Y)$ for all $X,Y \in \mathfrak{X} (S)$. As $S$ is non-minimal $H\neq 0$ and thus tr$A_1 \neq 0$. Therefore one finally arrives at
\be
\ii (X,Y) = \frac{1}{\mbox{tr}A_1} K_1 (X,Y) H \hspace{1cm} \forall X,Y \in \mathfrak{X} (S) 
\label{orto}
\ee
and thus $K_{\star^\perp H}=0$ from which $A_{\star^\perp H}=0$ follows.
\qed 
\end{proof}
\begin{remark}
\label{ortho=sub}
As a consequence, and due to Remark \ref{A-sub}, ortho-umbilical surfaces are trivial $A$-surfaces, and have {\em all} Weingarten operators proportional to each other with dim${\cal N}_1= 1$ (unless $S$ is totally geodesic, in which case of course dim${\cal N}_1= 0$). 
\qed 
\end{remark}
\begin{example}
Taking into account that the concept of ortho-umbilical $S$ seems to be new, I present some simple examples. Takle $\varietat = \mathbb{R} \times \Sigma$ for a 3-dimensional manifold $\Sigma$ and let $g=\mp dt^2\oplus g_{\Sigma_\pm}$ where $g_{\Sigma_\pm}$ is a Riemannian (+) or Lorentzian (-)  metric on $\Sigma$, so that $\espaitemps$ is a Lorentzian manifold. Now, take an arbitrary (spacelike) surface $S$ immersed in $\Sigma$. If $K$ is the second fundamental form of $S$ in $(\Sigma,g_{\Sigma_\pm})$ (with respect to the unit normal $m$ of $S$ in $(\Sigma,g_{\Sigma_\pm})$), then the shape tensor of $S$ in $\espaitemps$ can be easily shown to take the form $\ii (X,Y)= \pm K(X,Y) M$ for all $X,Y\in \mathfrak{X} (S)$, where $M\in \mathfrak{X}^\perp (S)$ is the normal that corresponds to $m$. The mean curvature vector $H$ is then proportional to $M$ and $S$ happens to be umbilical with respect to $\star^\perp M$, that is, ortho-umbilical (and also $M$-subgeodesic).
\qed
\end{example}
A surface can be pseudo-umbilical and ortho-umbilical at the same time. This can only happen in some special cases with a null $H$, to be enumerated and derived rigorously later in Remark \ref{pseudoandortho},  
or in the traditional cases of minimal surfaces or of totally umbilical surfaces, which is a particular case of the above and can be defined as:
\begin{definition}[Totally umbilical surfaces]
$S$ is called {totally umbilical} if it is umbilical with respect to all possible $N\in \mathfrak{X}(S)^\perp$:
$$
\forall N\in \mathfrak{X}(S)^\perp ,  \hspace{4mm} A_N = \frac{1}{2} g(H,N) \, {\bf 1} .
$$
Equivalently, 
\be
\ii (X,Y) = \frac{1}{2}  \bar g(X,Y) H , \hspace{1cm} \forall X,Y \in \mathfrak{X}(S) \, . \label{totumb}
\ee
\end{definition}
This provides a preliminary interpretation of the vector field $G$.
\begin{proposition}\label{prop:1}
Totally umbilical surfaces can be characterized by
$$
G=0 .$$
\end{proposition}
\begin{proof}
\smartqed
If $S$ is totally umbilical then $A_N =(1/2) g(H,N)\, {\bf 1}$ for all $N\in \mathfrak{X}(S)^\perp$, in particular $A_k=(1/2) g(H,k)\, {\bf 1}$ and $A_\ell =(1/2) g(H,\ell)\, {\bf 1}$ so that $\sigma_k =\sigma_\ell =0$ and thus $G=0$. Conversely, if $G=0$ then $\sigma_k =\sigma_\ell =0$ and thus
$$
K_\ell = \frac{1}{2} \mbox{tr} A_\ell\,Ê\,  \bar g \, \, ; \hspace{2cm} K_k = \frac{1}{2} \mbox{tr} A_k \,Ê\, \bar g .
$$
Using now formulas (\ref{shape}) and (\ref{mean}) one derives (\ref{totumb}).
\qed
\end{proof}
In what follows I am going to prove that, letting this case aside, the umbilical direction, if it exists, is always given by either $G$ or $\star^\perp G$ (Corollary \ref{coro:2}).

\section{Proof of the Main Theorems}
We are now ready to proof Theorems \ref{th:1} and \ref{th:2}. 
\begin{remark}
\label{points}
The results of the theorems can be stated at a point $x\in S$. For instance, ``a point $x\in S$ is $N$-umbilical if and only if two independent Weingarten operators commute at $x$''. However, for the sake of simplicity I am going to omit the sub-index $x$, and therefore the proofs are valid for the entire surface and in accordance with their form presented in the Introduction. One should keep in mind, though, that the result may be valid only at some points of the surface in general.
 \qed
 \end{remark}
\begin{proof}[of Theorem \ref{th:1}]
\smartqed

{\Large $\Longrightarrow$} \hspace{1cm} Assume that $N\in \mathfrak{X}(S)^\perp$ is an umbilical direction. In (say) the null basis $N=-g(N,k) \ell -g(N,\ell) k$ and the umbilical condition (\ref{umb}) can be written as
\be
-g(N,k)\,  A_\ell - g(N,\ell) \, A_k = \frac{1}{2} g(H,N) \, {\bf 1} \, .\label{umb2}
\ee
By taking here the commutator with $A_\ell$, or with $A_k$, one immediately derives (for $N\neq 0$):
$$
[A_\ell , A_k]=0 \, .
$$
Now, all possible Weingarten operators are linear combinations of any two of them, that is, for any $M\in \mathfrak{X}(S)^\perp$ there exist scalars $a$ and $b$ such that
$$A_M =a A_k + b A_\ell$$
and therefore 
$$
\left[A_M, A_{\tilde M}\right] =0, \hspace{4mm} \forall M,\tilde{M}\in \mathfrak{X}(S)^\perp \, .
$$

{{\Large $\Longleftarrow$}} \hspace{1cm} Conversely, assume that $[A_\ell , A_k]=0$. This implies that there exists a common ON eigen-basis such that both $A_\ell$ and $A_k$ are diagonal.
Let $\{\lambda_1,\lambda_2\}$ and $\{\nu_1 ,\nu_2\}$ denote the corresponding eigenvalues for $A_k$ and $A_\ell$, respectively.
Then, the equation (\ref{umb2}) to determine the umbilical direction $N$ becomes in this eigen-basis
\be
-g(N,\ell)\left(\begin{array}{cc}
\lambda_1 & 0 \\
0 & \lambda_2 \end{array} \right) - g(N,k) \left(\begin{array}{cc}
\nu_1 & 0 \\
0 & \nu_2 \end{array} \right) = \frac{1}{2}  g(H,N) \left(\begin{array}{cc}
1 & 0 \\
0 & 1 \end{array} \right) . \label{system}
\ee
Introducing here that
$$
g(H,N)=-g(N,\ell) \mbox{tr}A_k - g(N,k)\mbox{tr} A_\ell 
= -g(N,\ell)(\lambda_1 +\lambda_2)  - g(N,k)(\nu_1+\nu_2)
$$
the system of equations (\ref{system}) collapses to a \underline{single} equation
$$
g(N,\ell)(\lambda_1-\lambda_2) +g(N,k)(\nu_1 -\nu_2) =0.
$$
Its solution is clearly unique (up to proportionality factors) and explicitly given by $g(N,k)=-\lambda_1+\lambda_2$ and $g(N,\ell)=\nu_1 -\nu_2$, that is to say
\be
\fbox{$\displaystyle{
N_{umb} = (\lambda_1-\lambda_2)\ell - (\nu_1 -\nu_2) k} $}
\label{Numb}
\ee
unless $\lambda_1-\lambda_2=\nu_1 -\nu_2=0$, in which case the surface is totally umbilical as proven in the next Corollary \ref{coro:2}. 
\qed
\end{proof}
\begin{remark}
As a consequence, there exists a (generically unique) ON basis in which \underline{all} possible Weingarten operators diagonalize simultaneously.
\qed
\end{remark}
\begin{corollary}\label{coro:2}
The unique umbilical direction $N_{umb}$ is given, at each $x\in S$, either by $G|_x$ or $\star^\perp G|_x$.
\end{corollary}
\begin{proof}
\smartqed
It is straightforward to note that 
$$
(\lambda_1-\lambda_2)^2 = \sigma^2_k , \hspace{1cm} (\nu_1 -\nu_2)^2 =\sigma^2_\ell
$$
so that the unique solution (\ref{Numb}) for $N_{umb}$, {at each} $x\in S$, is either $\pm G|_{x}$ or $\pm \star^\perp G|_{x}$.
The only exceptional case is defined by $N_{umb} =G=0$, but this characterizes the totally umbilical case, as follows from Proposition \ref{prop:1}.
\qed
\end{proof}
Under the hypothesis of this corollary and Theorem \ref{th:1} one can also use the formula 
$G=\sigma_u u -\sigma_n n$ which does not hold in general. This is due to the commutativity property of all Weingarten operators in this case.

The causal character of the umbilical direction can be easily sorted out due to the explicit formula (\ref{Numb}), which allows us to compute
\be
g\left(N_{umb},N_{umb}\right)= 2(\lambda_1-\lambda_2)(\nu_1 -\nu_2) = 4\, \mbox{tr}(A_k A_\ell)-2\, \mbox{tr}A_k\, \mbox{tr}A_\ell . \label{causal}
\ee
Using here the expression (\ref{B}) for $B$, this can be invariantly rewritten as
$$
g\left(N_{umb},N_{umb}\right)=g(H,H)-2\, \mbox{tr}\,  B \, .
$$
Thus, the following criteria provide the causal character of the umbilical direction if it exists:
$$
g(H,H)-2\, \mbox{tr}\,  B \left\{\begin{array}{llc}
<0 & \Rightarrow N_{umb} & \mbox{is timelike} \\
>0 & \Rightarrow N_{umb} & \mbox{is spacelike}  \\
=0 & \Rightarrow N_{umb} & \mbox{is null}  \\
\end{array}
 \right.
$$
An alternative way of expressing the same utilizes the {\em ordered} eigen-bases for $A_\ell$ and $A_k$, where ordered means for instance that the first eigenvector corresponds to the larger eigenvalue. This has some relevance concerning the classification presented in \cite{S4}. Thus, from (\ref{causal})
 $$N_{umb}\,\,\, \mbox{is} \left\{\begin{array}{lll}
\mbox{spacelike} & \mbox{if the ordered eigen-bases of $A_\ell$ and $A_k$} & \mbox{agree}\\
\mbox{timelike} & \mbox{if the ordered eigen-bases of $A_\ell$ and $A_k$} & \mbox{are opposite}\\
\mbox{null} & \mbox{if one of the eigen-bases of $A_\ell$ or $A_k$} & \mbox{cannot be ordered}
\end{array}
\right.
$$

Let us now prove the second main theorem that, with the introduced notation, can be reformulated as:
\begin{remark}[{\bf Reformulation of Theorem \ref{th:2}}]
\label{reformulation}
{\em The necessary and sufficient condition for $S$ to be umbilical along a normal direction is 
$$
R^\perp (X,Y) N = \left(R(X,Y) N\right)^\perp , \hspace{1cm} \forall X,Y\in  \mathfrak{X}(S), \hspace{2mm} \forall N \in \mathfrak{X}(S)^\perp
$$
This is yet equivalent to 
\be
R(M,N,X,Y)= ds(X,Y)\, \,  g(\star^\perp N,M), \, \, 
\forall X,Y \in \mathfrak{X}(S), \hspace{2mm} \forall N,M \in \mathfrak{X}(S)^\perp \label{perp}
\ee
}
\end{remark}
\begin{proof}[of Theorem \ref{th:2}]
\smartqed
Using the Ricci equation (\ref{ricci}) and noting that $\left[A_M , A_N\right]=0$ due to Theorem \ref{th:1} one obtains (\ref{perp}).
Eliminating $M$ in this expression, one can also write
$$
(R (X,Y) N)^\perp = ds (X,Y) \star^\perp \! N ,\, \, \,  \forall X,Y\in  \mathfrak{X}(S), \hspace{2mm} \forall N \in \mathfrak{X}(S)^\perp
$$
which together with (\ref{normalR}) proves the result. 
\qed
\end{proof}
Finally, I give the proof of Corollary \ref{coro:1}.
\begin{proof}[of Corollary \ref{coro:1}]
\smartqed
We must prove that the necessary and sufficient condition for $S$ to be umbilical along a normal direction is that
$$R^\perp =0$$
for locally conformally flat spacetimes. It is well known \cite{Ei,Exact} that locally conformally flat semi-riemannian manifolds are characterized by the vanishing of the Weyl conformal curvature tensor $C$, defined by \cite{Ei}
\bean
C(v,w,y,z) \equiv R(v,w,y,z) +\frac{{\cal S}}{6}\left(g(v,y)g(w,z)-g(v,z)g(w,y) \right)\\
-\frac{1}{2}\left[Ric(v,y)g(w,z)-Ric(v,z)g(w,y)-Ric(w,y)g(v,z)+Ric(w,z) g(v,y) \right]
\eean
for all $v,w,y,z \in T\varietat$. It is then easily verified that in general
$$(R(X,Y)N)^\perp = (C(X,Y)N)^\perp , \hspace{1cm} \forall X,Y\in \mathfrak{X}(S), \hspace{2mm} \forall N\in \mathfrak{X}(S)^\perp$$
and consequently, if $\espaitemps$ is locally conformally flat, then 
$$(R(X,Y)N)^\perp = 0 , \hspace{1cm} \forall X,Y \in \mathfrak{X}(S), \hspace{2mm} \forall N\in \mathfrak{X}(S)^\perp$$
so that from Theorem \ref{th:2} one gets $ds =0$, or equivalently $$R^\perp =0.$$
\qed
\end{proof}

\section{Some important Corollaries and Consequences}
\label{sec:consequences}
In this section, I present several consequences of Theorems \ref{th:1} and \ref{th:2} for the special cases of pseudo-umbilical and ortho-umbilical surfaces.
\begin{corollary}[Pseudo-umbilical $S$]
A non-minimal spacelike surface $S$ is pseudo-umbilical if and only if $B$ is proportional to the Identity.
\end{corollary}
\begin{remark}
A more precise corollary will present the same statement at a point $x\in S$. I recall here Remark \ref{points} where this was carefully explained. For the sake of clarity, however, let me re-state now the previous corollary in its more precise version:

{\em At a non-minimal point $x\in S$, a spacelike surface $S$ is pseudo-umbilical if and only if $B|_x$ is proportional to the Identity.}

The same happens with all results in this paper. As the proof is always essentially the same as the one given, I will simply omit any further mention of this in what follows.
\qed
\end{remark}
\begin{proof} 
\smartqed
Assume that $S$ is pseudo-umbilical and $H\neq 0$. This means that $N_{umb}^\flat\wedge H^\flat =0$ which, on using expressions (\ref{mean}) for $H$ and (\ref{Numb}) for $N_{umb}$, becomes
$$
(\lambda_1-\lambda_2)(\nu_1+\nu_2) + (\lambda_1+\lambda_2)(\nu_1-\nu_2) =0$$
that is to say
$$
\lambda_1\nu_1-\lambda_2\nu_2 =0. $$
But then, on the common eigen-basis for $A_k$ and $A_\ell$ ---this eigen-basis does exist due to Theorem \ref{th:1}---, formula (\ref{B}) implies
$$
B=2\left(\begin{array}{cc}
-\lambda_1\nu_1 & 0 \\
0 & -\lambda_2 \nu_2
\end{array}
\right)
=-2\lambda_1\nu_1
\left(\begin{array}{cc}
1 & 0 \\
0 & 1
\end{array}
\right)
$$
or in other words
\be
B =\frac{1}{2} \mbox{tr} B \, {\bf 1} .\label{B=1}
\ee
Conversely, if (\ref{B=1}) holds then from (\ref{B})
$$
A_k A_\ell +A_\ell A_k =-\frac{1}{2} \mbox{tr} B \, {\bf 1} 
$$
and commuting here with $A_k$ and with $A_\ell$ one derives, respectively,
$$
\left[A_\ell , A_k^2\right]=0, \hspace{2cm} \left[A_k ,A_\ell^2\right]=0 .
$$
Using now the Cayley-Hamilton theorem ($A^2 -\mbox{tr}A \, A+\det A \, {\bf 1}=0$ for every $2\times 2$-matrix $A$) they become respectively
$$
\mbox{tr} A_k \, \left[A_\ell ,A_k\right]=0, \hspace{2cm} \mbox{tr} A_\ell \, \left[A_\ell ,A_k\right]=0
$$
so that $\left[A_\ell ,A_k\right]=0$ follows unless $\mbox{tr} A_k=\mbox{tr} A_\ell =0$, that is, unless $H=0$. Theorem \ref{th:1} then implies that (for $H\neq 0$) $S$ is umbilical along the direction (\ref{Numb}), and the calculation above (\ref{B=1}) can be reversed to check that this umbilical direction $N_{umb}$ is parallel to $H$.
\qed
\end{proof}

Note that the condition (\ref{B=1}) of the previous corollary can be invariantly characterized by 
$$(\mbox{tr} B)^2 - 4 \det B =0 .$$

\begin{corollary}[Ortho-umbilical $S$]
A non-minimal spacelike surface $S$ is ortho-umbilical if and only if
\be
\ii (X,Y)^\flat \wedge H^\flat =0 , \hspace{1cm} \forall X,Y \in \mathfrak{X} (S) \, .\label{ortocond0}
\ee
\end{corollary}
\begin{proof} 
\smartqed 
If $S$ is ortho-umbilical, from Remark \ref{ortho=sub} one knows that all Weingarten operators are proportional to each other so that, on using expression (\ref{orto}) one immediately derives (\ref{ortocond0}). Conversely, assume that (\ref{ortocond0}) holds (and $H\neq 0$). Then, there must exist a rank-2 symmetric covariant tensor field $\kappa$ on $S$ such that
\be
\ii (X,Y) =\kappa(X,Y) H, \hspace{1cm} \forall X,Y \in \mathfrak{X} (S) \label{ortocond}
\ee
and therefore $K_{\star^\perp H}(X,Y)=g(\star^\perp H ,\ii (X,Y))=0$ for arbitrary $X,Y \in \mathfrak{X} (S)$, that is to say, $K_{\star^\perp H}=0$, which leads to $A_{\star^\perp H}=0$.
\qed
\end{proof}
Observe that comparing (\ref{orto}) with (\ref{ortocond}) one has $\kappa = K_1 /{\mbox{tr}A_1}$. Defining $\tilde\kappa : \mathfrak{X} (S)\rightarrow \mathfrak{X} (S)$ by $g(\tilde{\kappa}(X),Y)=\kappa (X,Y)$ for all $X,Y \in \mathfrak{X} (S)$ it follows that
\be
\mbox{tr} \tilde\kappa =1 \, . \label{unittrace}
\ee
It is interesting to compare the totally umbilical condition (\ref{totumb}) with the more general ortho-umbilical one given by (\ref{ortocond}) together with (\ref{unittrace}).

The computation of $B$ for ortho-umbilical surfaces provides, by means of (\ref{ortocond}), the expression
$$
B= g(H,H)\, \tilde\kappa^2
$$
so that one has
\be
\mbox{tr} B = g(H,H) \mbox{tr}\tilde\kappa^2 =g(H,H) (1-2\det\tilde\kappa) \label{ortotrB}
\ee
where in the last step I have used, once more, the Caley-Hamilton theorem for $\tilde\kappa$ together with (\ref{unittrace}). Introducing the last formula in the Gauss equation (\ref{gauss2}) the following corollary follows.


\begin{corollary}
Ortho-umbilical surfaces satisfy the following relation between their Gaussian curvature, the curvature of the spacetime, and its normalized Lipschitz-Killing curvature $\det\tilde\kappa$:
\be
2\, K(S)= {\cal S} -4\,  \mbox{Ric} (\ell , k) +2 R(\ell,k,\ell,k) +2 g(H,H) \det\tilde\kappa \label{KS}
\ee
\end{corollary}
\begin{remark}
Recall that the Lipschitz-Killing curvature relative to $N \in \mathfrak{X} (S)$ is simply defined as $\det A_N$, see e.g. \cite{Rou0}. Givent that, for ortho-umbilical surfaces, all Weingarten operators are essentially the same and can be described up to proportionality factors by the unit-trace matrix $\tilde\kappa$, the concept of normalized Lipschitz-Killing curvature, represented by $\det\tilde\kappa$, makes sense and is well-defined.
\end{remark}
\begin{proof}
\smartqed

From the Gauss equation (\ref{gauss2}) and (\ref{ortotrB}) one gets (\ref{KS}) at once.
\qed
\end{proof}
\begin{corollary}
Ortho-umbilical surfaces in Lorentz space forms have vanishing normal curvature $R^\perp =0$  and they also satisfy the following relation between  the constant curvature ${\cal K}$ of $\espaitemps$, the Gaussian curvature of $S$ and its normalized Lipschitz-Killing curvature:
$$
K(S)={\cal K} + g(H,H) \det\tilde\kappa \, .$$
\end{corollary}
\begin{proof}
\smartqed
If $\espaitemps$ has constant curvature ${\cal K}$ it is in particular locally conformally flat so that Corollary \ref{coro:1} implies $R^\perp =0$. Then a trivial calculation using the constant-curvature hypotesis provides
$${\cal S} -4\,  \mbox{Ric} (\ell , k) +2 R(\ell,k,\ell,k)=2{\cal K}$$
so that (\ref{KS}) proves the result.
\qed
\end{proof}
\begin{remark}[{\bf The case when $S$ is pseudo- and ortho-umbilical}]
\label{pseudoandortho}
If a non-minimal $S$ is pseudo-umbilical as well as ortho-umbilical then, from the previous corollaries, clearly
$$
B= g(H,H)\, \tilde\kappa^2 \hspace{1cm} \mbox{and} \hspace{1cm} B=\frac{1}{2} \mbox{tr} B\,  \bf{1}
$$
This actually implies that either
\begin{enumerate}
\item $2\tilde\kappa = \bf{1}$, that is $2\kappa=\bar g$, so that from (\ref{ortocond}) $S$ is actually totally umbilical, or
\item 
$$
g(H,H) =0, \hspace{1cm} B=0
$$
so that 
they have a shape tensor of the form
$$
\ii (X,Y)=-K_k(X,Y) \, \ell \, \, \, \, \, \mbox{or} \,\,\, -K_\ell(X,Y) \, k \hspace{1cm} \forall X,Y \in \mathfrak{X} (S)
$$
Thus, they are $H$-subgeodesic MOTS and also $0$-isotropic in the sense of \cite{CFG}. \qed
\end{enumerate}
\end{remark}


Let me finally state an instance where there always exist umbilical-type surfaces. Consider a space-time with an integrable conformal Killing vector $\xi$ (no causal character for $\xi$ is required nor necessary here)\cite{Ei,Exact}, that is, such that
\be
\forall v,w\in T\varietat, \hspace{1cm} 2\phi\,  g(v,w)=g(\nabla_v\xi , w) +g(\nabla_w\xi, v) 
\label{killing}
\ee
and also 
$$
\xi^\flat \wedge d\xi^\flat =0 .
$$
This last condition implies that $\xi$ is orthogonal to an integrable distribution, in other words, locally there exist functions $F$ and $\tau$ such that $\xi^\flat = F d\tau$, hence $\tau=$const. is a family of hypersurfaces orthogonal to $\xi$. Consider any spacelike surface $S$ imbedded in any of these orthogonal hypersurfaces (such that $\xi|_S \neq 0$). Then, $\xi\in \mathfrak{X} (S)^\perp$ and one can define its Weingarten operator $A_\xi$. From (\ref{killing}) one has
\bean
\forall X,Y\in \mathfrak{X}(S), \hspace{1cm} 2\phi|_S \,  g(X,Y)=g(\nabla_X\xi , Y) +g(\nabla_Y\xi, X)= \\
= -g(A_\xi (X),Y)- g(A_\xi (Y), X)  = -2 g(A_\xi (X),Y)\hspace{1cm}
\eean
ergo
$$
A_\xi = -\phi|_S \, \bm{1} .
$$
Thus, any such $S$ is $\xi$-umbilical and it satisfies all the properties shown above for them.

\section{Final Considerations}
Even though this paper has focused on spacelike surfaces in Lorentzian 4-dimensional manifolds, the concepts and ideas can also be considered in other dimensions and signatures, and for other types of surfaces. As a matter of fact, the main result of this paper, the commutativity of the Weingarten operators for umbilical-type surfaces, holds true, {\it mutatis mutandis}, for spacelike surfaces in 4-dimensional semi-Riemannian manifolds of {\em arbitrary signature}. The theorems are also valid for {\em timelike} surfaces. In both generalizations one only has to rewrite the proofs in ON bases (leaving the appropriate signs free to cover all possibilities).

Unfortunately, the result is exclusive, however, of dimension four and co-dimension two. A simple analysis shows that
\begin{enumerate}
\item co-dimension two spacelike submanifolds in semi-Riemannian manifolds of {\em higher dimensions} will also have \underline{two} independent Weingarten operators, and their commutativity at a point can be seen to be a necessary condition for the point to be umbilic. However, it cannot be sufficient in general. For $n$-dimensional manifolds the problem resides in the fact that any Weingarten operator is an $(n-2)\times (n-2)$ matrix, so that in diagonal form the number of equations to determine a relation between the two independent components of the would-be umbilical direction $N_{umb}$ is too large, and has no solution in general.
\item If the co-dimension is greater than two, then there are \underline{more than two} independent Weingarten operators, and their commutativity is not even a necessary condition, as can be easily checked. There can be a linear combination of three or more matrices which is proportional to the identity while the matrices do not commute.
\end{enumerate}
It will be interesting to know if there are any generalizations of the results in this paper to arbitrary dimension.

\begin{acknowledgement}
I thank Miguel S\'anchez and the referees for some comments.
Supported by grants
FIS2010-15492 (MICINN), GIU06/37 (UPV/EHU) and P09-FQM-4496 (J. Andaluc\'{\i}a---FEDER).
\end{acknowledgement}

\end{document}